\documentclass[10pt,oneside]{amsart}

\usepackage{latexsym,amssymb,amsmath,amscd,amsxtra,amsfonts,amsthm}
\usepackage{fancyhdr,array,dsfont}

\newtheorem{theorem}{Theorem}[section]
\newtheorem{lemma}[theorem]{Lemma}

\newtheorem{remark}[theorem]{Remark}

\newtheorem{question}{Question}

\newcommand{\beq}{\begin{eqnarray*}}
\newcommand{\eeq}{\end{eqnarray*}}
\newcommand{\beqn}{\begin{eqnarray}}
\newcommand{\eeqn}{\end{eqnarray}}

\numberwithin{equation}{section}

\begin{document}





\title[Characterization of transportation-information inequalities]{A new characterization of quadratic transportation-information inequalities}

\author[Y. LIU]{Yuan LIU}
\address{Yuan LIU, Institute of Applied Mathematics, Academy of Mathematics and Systems Science,
Chinese Academy of Sciences, Beijing 100190, China}
\email{liuyuan@amss.ac.cn}

\date{\today}

\begin{abstract}
It is known that a quadratic transportation-information inequality $\mathrm{W\hspace*{-0.5mm}_2I}$ interpolates between the Talagrand's inequality $\mathrm{W\hspace*{-0.5mm}_2H}$ and the log-Sobolev inequality (LSI for short). The aim of this paper is threefold:
\begin{enumerate}
\item To prove the equivalence of $\mathrm{W\hspace*{-0.5mm}_2I}$ and the Lyapunov condition, which gives a new characterization inspired by Cattiaux-Guillin-Wu \cite{CGW}.
\item To prove the stability of $\mathrm{W\hspace*{-0.5mm}_2I}$ under bounded perturbations, which gives a transference principle in the sense of Holley-Stroock.
\item To prove $\mathrm{W\hspace*{-0.5mm}_2H}$ through a restricted $\mathrm{W\hspace*{-0.5mm}_2I}$, which gives a counterpart of the restricted LSI presented by Gozlan-Roberto-Samson \cite{GRS}.
\end{enumerate}
\end{abstract}



\maketitle

\allowdisplaybreaks

\section{Introduction}
 \label{Intro}
 \setcounter{equation}{0}

Transport inequalities have been a very active family of functional inequalities for years, with their profound connections to measure concentration phenomenon and large deviation principle of Markov processes. We refer to two monographies by Villani \cite{Villani1, Villani2} and two surveys by Gozlan-L\'{e}onard \cite{GL-survey} and Cattiaux-Guillin \cite{CG} on this subject with references therein. In this paper, we will investigate three questions about the quadratic transportation-information inequality $\mathrm{W\hspace*{-0.5mm}_2I}$, which interpolates between the Talagrand's inequality $\mathrm{W\hspace*{-0.5mm}_2H}$ and the log-Sobolev inequality (LSI for short).

Let's address some basics around the above objects. Denote by $(E,d)$ a metric space equipped with the collection $\mathcal{P}(E)$ of all probability measures on the Borel $\sigma$-field. Define the $L^p$ Wasserstein (transportation) distance between two probability measures $\nu, \mu\in \mathcal{P}(E)$, for any $p\geqslant 1$, by
    \[ W_p(\nu,\mu) = \left(\inf\limits_{\pi\in \mathcal{C}(\nu,\mu)} \int_{E\times E} d^p(x,y) \pi(\mathrm{d}x, \mathrm{d}y)\right)^{1/p}, \]
where $\mathcal{C}(\nu,\mu)$ denotes the set of couplings between $\nu$ and $\mu$, that is to say the set of probability measures $\pi$ on $E\times E$ with marginals $\nu$ and $\mu$. For simplicity, our framework is specified as the following. Take $E$ to be a connected complete Riemannian manifold of finite dimension, $d$ the geodesic distance, $\mathrm{d}x$ the volume measure, $\mu(\mathrm{d}x) = e^{-V(x)}\mathrm{d}x \in \mathcal{P}(E)$ with $V\in C^1(E)$, $\mathrm{L}=\Delta - \nabla V\cdot \nabla$ the $\mu$-symmetric diffusion operator with domain $\mathbb{D}(\mathrm{L})$, $\Gamma(f,g) = \nabla f\cdot \nabla g$ the carr\'{e} du champ operator and $\mathcal{E}$ the Dirichlet form with domain $\mathbb{D}(\mathcal{E})$. Denote $\mu(h) := \int h \mathrm{d}\mu$. It's known that the integration by parts formula reads
   \[ \int \nabla f \cdot \nabla g \;\mathrm{d}\mu = -\int f \mathrm{L}g  \;\mathrm{d}\mu, \ \forall f\in \mathbb{D}(\mathcal{E}), g\in \mathbb{D}(\mathrm{L}). \]
The reader is referred to Bakry-Gentil-Ledoux \cite{BGL} for a detailed presentation.
  
In the sequel, we focus on the case $p=2$ of particular interest. Twenty years ago, Talagrand \cite{Talagrand} introduced the celebrated transportation-entropy inequality for any $\nu$ with respect to the Gaussian measure $\mu$ (hence, it is called the Talagrand's inequality)
  \begin{align}
    W_2(\nu,\mu)\leqslant \sqrt{2C H(\nu|\mu)}, \tag{$\mathrm{W\hspace*{-0.5mm}_2H(C)}$}
  \end{align}
where $H(\nu|\mu)$ denotes the relative entropy equal to either $\nu (\log\frac{\mathrm{d}\nu}{\mathrm{d}\mu})$ if $\nu$ is absolutely continuous to $\mu$ or $\infty$ otherwise. For general $\mathrm{W\hspace*{-0.5mm}_2H}$, various characterizations have been found out, such as the Bobkov-G\"{o}tze's infimum-convolution criterion (see \cite{BG}), the Gozlan-L\'{e}onard's large deviation and concentration criteria (see \cite{GL}), etc.

As a counterpart, the transportation-information inequality was introduced much later by Guillin-L\'{e}onard-Wu-Yao \cite{GLWY}, which substituted Fisher-Donsker-Varadhan information for relative entropy, i.e.
  \begin{align}
    W_2(\nu,\mu)\leqslant \sqrt{2C I(\nu|\mu)}, \tag{$\mathrm{W\hspace*{-0.5mm}_2I(C)}$}
  \end{align}
where $I(\nu|\mu)$ is associated to some Dirichlet form $\mathcal{E}$ with domain $\mathbb{D}(\mathcal{E})$ in $L^2(\mu)$ as
  \[ I(\nu|\mu) = \left\{ \begin{array}{ll} \mathcal{E}(\sqrt{f}, \sqrt{f}),& \textrm{ if } f=\frac{\mathrm{d}\nu}{\mathrm{d}\mu} \textrm{ with }  \sqrt{f}\in \mathbb{D}(\mathcal{E});\\  \infty,& \textrm{ otherwise.} \end{array} \right. \]
Some criteria for $\mathrm{W\hspace*{-0.5mm}_2I}$ have also been worked out correspondingly by \cite{GLWY}.

The reference measure $\mu$ is usually regarded as the unique invariant distribution of (symmetric) diffusion process. However in some practical cases, $\mu$ is unknown except its existence. For this reason, it is natural to present  suitable hypotheses on the infinitesimal generator instead of its equilibrium limit. Going to this direction, Cattiaux-Guillin-Wang-Wu \cite{CGWW} drew in the Lyapunov condition to study the super Poincar\'{e} inequalities.  Afterwards, Cattiaux-Guillin-Wu \cite{CGW} derived $\mathrm{W\hspace*{-0.5mm}_2H}$ from the Lyapunov condition, which  even worked for LSI
  \[ H(\nu|\mu)\leqslant 2CI(\nu|\mu)\]
with additional assumptions on the Bakry-Emery's curvature. 

More precisely as \cite{CGW}, say $W\geqslant 1$ is a Lyapunov function if there exist two constants $c>0, b\geqslant 0$ and some $x_0\in E$ such that in the sense of distribution
  \beqn
    \mathrm{L}W \leqslant \left(-cd^2(x,x_0) + b\right)W. \label{eqLya0}
  \eeqn
There were also some variants of (\ref{eqLya0}) by substituting $d^2(\cdot,\cdot)$ with other functionals of distance, such as \cite{BCG, CGZ} for investigating relations between the Poincar\'{e} inequalities and weak Lyapunov conditions.

\begin{theorem} (\cite[Theorem 1.2]{CGW}) \label{thmW2H}
Under the Lyapunov condition (\ref{eqLya0}),
\begin{enumerate}
\item There exists a constant $C_1>0$ such that $\mu$ verifies the inequality $\mathrm{W\hspace*{-0.5mm}_2H}(C_1)$.

\item There exists a constant $C_2>0$ such that $\mu$ verifies the inequality $\mathrm{LSI}(C_2)$ provided that the Bakry-Emery's curvature is bounded from below.
\end{enumerate}
\end{theorem}

In the preprint \cite{Liu}, the author tried to show that if the Bakry-Emery's curvature has a lower bound, the LSI and Lyapunov condition (with a slight adjustment) are equivalent. On the other hand, due to the well-known HWI inequality from Otto-Villani \cite{Otto-Villani}, it follows that  LSI and $\mathrm{W\hspace*{-0.5mm}_2I}$ are also equivalent under the same curvature condition. These results lead to our first question.

\begin{question} \label{ques1}
Is it possible to get $\mathrm{W\hspace*{-0.5mm}_2I}$ through the Lyapunov condition without any assumption on the curvature? And what about the converse implication?
\end{question}

For some technical reasons, certain regularity condition is usually imposed on the density of $\mu$. One way of removing such a constraint is to handle a nice model space firstly and then extend the associated functional inequalities to a general setting through perturbations. For example, Holley-Stroock \cite{HS} proved that if $\mu$ satisfies a LSI and $\tilde{\mu}$ is a bounded variant of $\mu$ with $\frac{\mathrm{d}\tilde{\mu}}{\mathrm{d}\mu}$ and $\frac{\mathrm{d}\mu}{\mathrm{d}\tilde{\mu}} \in L^\infty$, then $\tilde{\mu}$ verifies a LSI too. This is called the stability or transference principle of inequalities.

For LSI, perturbations on the entropy and information can be respectively and directly controlled. Somehow it is tough to deal with the Wasserstein distance, so that the stability of $\mathrm{W\hspace*{-0.5mm}_2H}$ remained open until it was attacked by Gozlan-Roberto-Samson \cite{GRS} via a new characterization, the so called restricted LSI (rLSI for short). They defined $f$ to be a $K$-semi-convex function if for any $x,y\in \mathbb{R}^n$ (which can be extended to manifolds or length spaces, etc.)
  \[ f(y) \geqslant f(x) + \nabla f(x) \cdot (y-x) - \frac{K}{2}d^2(x,y), \]
see also the classical semi-convexity in the textbook of Evans \cite[Section 3.3]{Evans}. Say $\mu$ verifies a rLSI with constant $C>0$ if for all $K$-semi-convex $f$ with $0\leqslant K< C^{-1}$
  \begin{align}
    \mathrm{Ent}_\mu(e^f)\leqslant \frac{2C}{(1-KC)^2} \int |\nabla f|^2e^f \mathrm{d}\mu, \tag{$\mathrm{rLSI(C)}$}
  \end{align}
where $\mathrm{Ent}_\mu(e^f) = \mu(fe^f) - \mu (e^f) \log\mu (e^f)$, and actually $\int |\nabla f|^2e^f \mathrm{d}\mu = 4I(e^f)$. We quote partial results from \cite{GRS} as follows.

\begin{theorem} (\cite[Theorem 1.5]{GRS}) \label{thmW2HStable}
The next two statements are equivalent:
\begin{enumerate}
\item There exists a constant $C_1>0$ such that $\mu$ verifies the inequality $\mathrm{W\hspace*{-0.5mm}_2H}(C_1)$.

\item There exists a constant $C_2>0$ such that $\mu$ verifies the inequality $\mathrm{rLSI(C_2)}$.
\end{enumerate}
A quantitative relationship between control constants is the following: from $(1)$ to $(2)$ holds $C_1=C_2$, and conversely, from $(2)$ to $(1)$ holds $C_1=8C_2$.
\end{theorem}

Now our second question arises.
\begin{question} \label{ques2}
What about the stability of $\mathrm{W\hspace*{-0.5mm}_2I}$ in the sense of Holley-Stroock?
\end{question}

It is known that $\mathrm{W\hspace*{-0.5mm}_2I}$ implies $\mathrm{W\hspace*{-0.5mm}_2H}$ according to Guillin-L\'{e}onard-Wang-Wu \cite[Theorem 2.4]{GLWW}, but it is not clear yet to what extent they are different. Moerover, $\mathrm{W\hspace*{-0.5mm}_2I}$ is an interpolation between $\mathrm{W\hspace*{-0.5mm}_2H}$ and LSI, which suggests that this relationship should be kept for the restricted type inequalities. So our third question comes out.
\begin{question} \label{ques3}
What is the restricted $\mathrm{W\hspace*{-0.5mm}_2I}$ equivalent to $\mathrm{W\hspace*{-0.5mm}_2H}$?
\end{question}

\bigskip

This paper will answer the above three questions. With a slight adjustment to (\ref{eqLya0}), say $W>0$ is a Lyapunov function if $W^{-1}$ is locally bounded and there exist two constants $c>0, b\geqslant 0$ and some $x_0\in E$ such that in the sense of distribution
  \beqn
    \mathrm{L}W \leqslant \left(-cd^2(x,x_0) + b\right)W. \label{eqLya}
  \eeqn
Here we use $W>0$ instead of $W\geqslant 1$ in (\ref{eqLya0}), but the technique in the proof of Bakry-Barthe-Cattiaux-Guillin \cite[Theorem 1.4]{BBCG} still works if $W^{-1}$ is locally bounded (not necessary to request a uniform lower bound). We will go back to this point in Section 2.

What we prove for the first question is the following.
\begin{theorem} \label{thmW2ILya}
The next two statements are equivalent:
\begin{enumerate}
\item There exists a constant $C>0$ such that $\mu$ verifies the inequality $\mathrm{W\hspace*{-0.5mm}_2I}(C)$.

\item There exists a Lyapunov function $W>0$ that verifies Condition (\ref{eqLya}).
\end{enumerate}
\end{theorem}

On the second question, our method is not to look for an equivalent restricted type LSI for $\mathrm{W\hspace*{-0.5mm}_2I}$ , but turns to applying Theorem \ref{thmW2ILya} and \ref{thmW2HStable} directly.
\begin{theorem} \label{thmW2IStable}
Let $\tilde{\mu}\in \mathcal{P}(E)$ be absolutely continuous to $\mu$ with $M^{-1}\leqslant \frac{\mathrm{d}\tilde{\mu}}{\mathrm{d}\mu} \leqslant M$ for some constant $M\geqslant 1$. Then $\tilde{\mu}$ verifies the inequality $\mathrm{W\hspace*{-0.5mm}_2I}$ if so does $\mu$.
\end{theorem}

Similarly to \cite{GRS}, we introduce if $\log\frac{\mathrm{d}\nu}{\mathrm{d}\mu}$ is $K$-semi-convex with $0\leqslant K< C^{-1}$
  \begin{align}
    W_2(\nu,\mu)\leqslant \sqrt{\frac{4C^2}{(1-KC)^2} I(\nu|\mu)}. \tag{$\mathrm{rW\hspace*{-0.5mm}_2I(C)}$}
  \end{align}
The above constant is chosen for convenience. For the third question, we prove

\begin{theorem} \label{thmRestrictedW2I}
The next two statements are equivalent:
\begin{enumerate}
\item There exists a constant $C_1>0$ such that $\mu$ verifies the inequality $\mathrm{W\hspace*{-0.5mm}_2H}(C_1)$.

\item There exists a constant $C_2>0$ such that $\mu$ verifies the inequality $\mathrm{rW\hspace*{-0.5mm}_2I}(C_2)$.
\end{enumerate}
A quantitative relationship between control constants is the following: from $(1)$ to $(2)$ holds $2C_1=C_2$, and conversely, from $(2)$ to $(1)$ holds $C_1=4C_2$.
\end{theorem}

The rest of this paper contains three sections, which give the proofs successively for Theorems \ref{thmW2ILya}, \ref{thmW2IStable} and \ref{thmRestrictedW2I}. An alternative proof of $\mathrm{W\hspace*{-0.5mm}_2H}$ through the Lyapunov condition (\ref{eqLya}) is also provided, see Section 4.

\bigskip

\section{Equivalence of $\mathrm{W\hspace*{-0.5mm}_2I}$ and the Lyapunov condition}
 \label{W2ILya}
 \setcounter{equation}{0}
Throughout this paper, write $d_0^2(x) := d^2(x,x_0)$. The use of Lyapunov condition (\ref{eqLya}) is based on \cite[Theorem 1.4]{BBCG}.

\begin{lemma}\label{lemLya}
Under (\ref{eqLya}), for any $h\in \mathbb{D}(\mathcal{E})$
   \beqn
      \int h^2 d_0^2 \mathrm{d}\mu \leqslant \frac{1}{c} \int |\nabla h|^2 \mathrm{d}\mu + \frac{b}{c}\int h^2 \mathrm{d}\mu. \label{eqTransLya}
   \eeqn
Moreover, the Poincar\'{e} inequality holds, i.e. there exists $C>0$ such that
   \begin{align}
     \int |h- \mu (h)|^2 \mathrm{d}\mu \leqslant C \int |\nabla h|^2 \mathrm{d}\mu. \tag{$\mathrm{PI}(C)$}
   \end{align}
\end{lemma}
\begin{proof}
The technique in \cite[Page 64]{BBCG} yields
  \beq
    \int h^2 d_0^2 \mathrm{d}\mu
    &=& \frac{1}{c} \int h^2(cd_0^2-b) \mathrm{d}\mu + \frac{b}{c}\int h^2 \mathrm{d}\mu \\
    &\leqslant& \frac{1}{c} \int \frac{-\mathrm{L}W}{W} h^2 \mathrm{d}\mu + \frac{b}{c}\int h^2 \mathrm{d}\mu \\
    &=& \frac{1}{c} \int \nabla W \cdot \nabla \frac{h^2}{W}  \mathrm{d}\mu + \frac{b}{c}\int h^2 \mathrm{d}\mu \\
    &=& \frac{1}{c} \int |\nabla h|^2 - \left| \nabla h - \frac{h}{W}\nabla W \right|^2  \mathrm{d}\mu + \frac{b}{c}\int h^2 \mathrm{d}\mu \\
    &\leqslant& \frac{1}{c} \int |\nabla h|^2 \mathrm{d}\mu + \frac{b}{c}\int h^2 \mathrm{d}\mu.
  \eeq
Note that here is no need to assume the integrability of $d_0^2$ or $\frac{-\mathrm{L}W}{W}$ for $\mu$, since we can take an approximation sequence in $C_c^\infty(E)$ for given $h$.

Next, according to \cite{BBCG}, let $B$ be a ball of radius $r$ centered at $x_0$ with $r^2 = \frac{b}{c}+1$. Denote $\bar{h}_{B} = \mu(h\mathbf{1}_B)/\mu(B)$, we have
  \beq
     \int |h-\mu(h)|^2 \mathrm{d}\mu &=& \inf\limits_{\lambda \in \mathbb{R}} \int |h-\lambda|^2 \mathrm{d}\mu 
     \ \leqslant\   \int |h-\bar{h}_{B} |^2 \mathrm{d}\mu \\
     &=& \int |h-\bar{h}_{B} |^2 d_0^2 \mathrm{d}\mu  +  \int |h-\bar{h}_{B} |^2 (1-d_0^2) \mathrm{d}\mu\\
     &\leqslant& \frac{1}{c} \int |\nabla h|^2 \mathrm{d}\mu + \left(\frac{b}{c}+1\right)\int_{B} |h-\bar{h}_{B} |^2 \mathrm{d}\mu, 
  \eeq
 which implies the (global) Poincar\'{e} inequality by combining the local one on $B$.
\end{proof}
\begin{remark}
Under (\ref{eqLya0}), it was proved that $\mu$ admits the Gaussian integrability $\mu e^{\delta d_0^2} < \infty$ for some $\delta>0$. In \cite{Liu0}, an elementary proof of this fact was given with a sharp estimate for $\delta$. Actually, these all hold under (\ref{eqLya}) too.
\end{remark}

Recall Barthe-Cattiaux-Roberto \cite[Lemma 14]{BCR}, the following inequality holds.

\begin{lemma} \label{lemBarthe}
For any $s\in [0, Nt]$ with $N\geqslant0$ and $t>0$,
   \[ s^2\log\frac{s^2}{t^2} - (s^2 - t^2) \leqslant (1+N)^2(s-t)^2. \]
\end{lemma}
\begin{proof}
Use the Lagrange's mean-value theorem to get $\log a - \log b \leqslant \frac{a-b}{b}$, and then substitute  $\frac{s^2}{t^2}$ for $a$ and $1$ for $b$.
\end{proof}

Recall the definition of infimum-convolution $Q_th$ is for any $t\geqslant0$
   \[ Q_th(x) := \inf\limits_{y\in E} \left\{ h(y) + \frac12d^2(x,y) \right\} . \]
\begin{lemma} \label{lemConvolution}
Let $g=\frac12 (d_0^2 \wedge D)$ for any constant $D\in [0,\infty]$. Then $Q_1 g \geqslant \frac12 g$.
\end{lemma}
\begin{proof}
The definition gives
  \[ Q_1 g(x) = \inf\limits_{y} \left\{ g(y) + \frac12 d^2(x,y) \right\} = \frac12 \inf\limits_{y} \left\{ d^2(y,x_0) \wedge D + d^2(x,y) \right\}. \]
Suppose the infimum is achieved at some $z$, it follows either $d^2(z,x_0) < D$ so that
   \[ d^2(z,x_0) \wedge D + d^2(x,z) = d^2(z,x_0) + d^2(x,z) \geqslant \frac12 d^2(x,x_0) \geqslant \frac12 \left(d^2(x,x_0)  \wedge D\right), \] 
or $d^2(z,x_0) \geqslant D$ so that 
   \[ d^2(z,x_0) \wedge D + d^2(x,z) = D  + d^2(x,z) \geqslant D \geqslant  d^2(x,x_0)  \wedge D. \]
Hence, it is always true that
      \[ Q_1 g(x) \geqslant \frac14 \left( d^2(x,x_0)  \wedge D\right) \ =\ \frac12 g(x). \]
The proof is completed.
\end{proof}

Now we prove Theorem \ref{thmW2ILya}.

\begin{proof} The strategy contains two parts.

{\bf Part 1}. Assume the Lyapunov condition (\ref{eqLya}) holds. Then $\mathrm{W\hspace*{-0.5mm}_2H}(C_T)$ comes true with constant $C_T>0$ by \cite[Theorem 1.9]{CGW}, which also verifies the Bobkov-G\"{o}tze's  infimum-convolution criterion in \cite{BG} that for any bounded $h$ with $\mu (h) =0$,
    \beqn
      \int e^{Q_{C_T}h} \mathrm{d}\mu \leqslant 1. \label{eqBobGot}
    \eeqn

From the representation (see for example Bakry-Gentil-Ledoux \cite[Section 9.2]{BGL})
   \[ \frac12 W_2(\nu,\mu)^2 \ = \sup_{h\ \textrm{bounded}} \int Q_1h \mathrm{d}\nu - \int h\mathrm{d}\mu
            \ = \sup_{\scriptsize\begin{array}{c} h\ \textrm{bounded}\\ \mu (h) =0 \end{array}} \int Q_1h \mathrm{d}\nu, \]
it follows
   \[ \frac1{2C_T} W_2(\nu,\mu)^2 = \sup_{\scriptsize\begin{array}{c} h\ \textrm{bounded}\\ \mu (h) =0 \end{array}} \int Q_{C_T} h \mathrm{d}\nu. \]
Let $\mathrm{d}\nu = f^2\mathrm{d}\mu$ with $f>0$. We introduce a subset with some parameter $N>1$
   \[ A = \{x: 0< f\leqslant N\cdot \mu (f) \},  \]
to make the following decomposition
   \[ \int f^2Q_{C_T} h\mathrm{d}\mu = \int_{A} f^2(Q_{C_T} h - \log f^2) \mathrm{d}\mu
         + \int_{A} f^2 \log f^2 \mathrm{d}\mu + \int_{A^c} f^2 Q_{C_T} h \mathrm{d}\mu. \]

Now we estimate each term in the above sum. First of all, using the inequality $\log a - \log b \leqslant \frac{a-b}{b}$ and (\ref{eqBobGot}) yields
   \beqn
      \int_{A} f^2(Q_{C_T} h - \log f^2) \mathrm{d}\mu
      &\leqslant& \int_{A} e^{Q_{C_T} h} - f^2 \mathrm{d}\mu \nonumber\\
      &=& \int e^{Q_{C_T} h} - f^2 \mathrm{d}\mu + \int_{A^c} f^2 - e^{Q_{C_T} h}\mathrm{d}\mu \nonumber\\
      &\leqslant& \int_{A^c} f^2\mathrm{d}\mu \ \leqslant\ \frac{N^2}{(N-1)^2} \int_{A^c} \left(f-\mu (f)\right)^2 \mathrm{d}\mu.
      \label{eqDecom1}
   \eeqn
The last inequality is due to $f> N\cdot \mu (f)$ on $A^c$.

Next, since $0<\left(\mu (f)\right)^2 \leqslant \mu (f^2) =1$, we have by Lemma \ref{lemBarthe} for $t=\mu (f)$ and $s=f(x)$ with $x\in A$ that
   \beqn
      \int_{A} f^2\log f^2 \mathrm{d}\mu
      &=& \int_{A} f^2\log \frac{f^2}{\left(\mu (f)\right)^2} \mathrm{d}\mu + \int_{A} f^2 \mathrm{d}\mu \cdot \log\left(\mu (f)\right)^2 \nonumber\\
      &\leqslant& \int_{A} f^2 - \left(\mu (f)\right)^2 + (1+N)^2 \left(f- \mu (f)\right)^2 \mathrm{d}\mu \nonumber \\
      &\leqslant& \int f^2 - \left(\mu (f)\right)^2 \mathrm{d}\mu + (1+N)^2 \int_{A} \left(f- \mu (f)\right)^2 \mathrm{d}\mu \nonumber\\
      &\leqslant& \left[1+(1+N)^2\right] \int \left(f- \mu (f)\right)^2 \mathrm{d}\mu. \label{eqDecom2}
   \eeqn

Thirdly, the definition of infimum-convolution gives
  \[ Q_{C_T} h(x) \leqslant \int h(y) + \frac{1}{2C_T}d^2(x,y)\mathrm{d}\mu(y) \leqslant  \frac{1}{C_T} \left( d_0^2(x)+\mu (d_0^2) \right), \]
and then due to Lemma \ref{lemLya}
   \beqn
      \int_{A^c} f^2Q_{C_T} h \; \mathrm{d}\mu
      &\leqslant& \frac{N^2}{C_T(N-1)^2} \int_{A^c} \left(f-\mu (f) \right)^2 \left( d_0^2+\mu (d_0^2) \right) \mathrm{d}\mu \nonumber\\
      &\leqslant& \frac{N^2}{C_T(N-1)^2} \left(\frac{1}{c} \int |\nabla f|^2\mathrm{d}\mu +  \frac{b+\mu (d_0^2)}{c} \int \left(f-\mu (f)\right)^2 \mathrm{d}\mu \right). \label{eqDecom3}
   \eeqn

Combining (\ref{eqDecom1}-\ref{eqDecom3}) and using the Poincar\'{e} inequality $\mathrm{PI}(C_P)$ yield
   \[ W_2(\nu,\mu) \leqslant \sqrt{2C I(\nu|\mu)}, \]
where $C$ is less than $C_TC_P[\frac{N^2}{(N-1)^2}+1+(1+N)^2] + \frac{N^2}{c(N-1)^2}\left[1+\left(b+\mu (d_0^2)\right)C_P\right]$.

\bigskip
{\bf Part 2}. Assume $\mathrm{W\hspace*{-0.5mm}_2I}(C)$ holds, which implies $\mathrm{W\hspace*{-0.5mm}_2H}(C)$ by \cite[Theorem 2.4]{GLWY} and then $\mathrm{W\hspace*{-0.5mm}_1H}(C)$ automatically. According to Djellout-Guillin-Wu \cite[Theorem 2.3]{DGW}, $\mu$ satisfies the Gaussian integrability $\mu (e^{\delta d_0^2}) < \infty$ for some $\delta>0$.

A Lyapunov function can be constructed by solving certain elliptic equation. Set $\phi = -cd_0^2 + b$ with two parameters $c>0, b\geqslant 0$, which will be determined below. Introduce a partial differential equation of second order for $w\in L^2(\mu)$
  \beqn
     \mathrm{H}u := -\mathrm{L}u + \phi u = w. \label{eqSchr}
  \eeqn

First of all, (\ref{eqSchr}) gives $\mu(u\cdot \mathrm{H}u) \leqslant \mathcal{E}[u] + b \mu (u^2)$ directly. On the other hand, for any $u\in \mathbb{D}(\mathcal{E})$ with $\mu (u^2)>0$, let $\mathrm{d}\nu = u^2/\mu (u^2) \mathrm{d}\mu$ and $g_D = \frac12 (d_0^2\wedge D)$, using Lemma \ref{lemConvolution} and $\mathrm{W\hspace*{-0.5mm}_2I}$ yields
   \beq
      \mu\left(u^2 d_0^2\right) &=& \lim\limits_{D\to \infty} \mu\left(u^2 (d_0^2 \wedge D)\right)
      \ \leqslant\ \lim\limits_{D\to \infty}4\mu\left(u^2 Q_1g_D\right)\\
      &\leqslant& \lim\limits_{D\to \infty}4\mu (u^2) W_2(\nu,\mu)^2 + 4\mu (u^2) \mu (g_D)
      \ \leqslant\ 8C\mathcal{E}[u] + 2\mu (u^2) \mu (d_0^2),
   \eeq
which implies by taking $c=1/(16C)$ and $b=4c\cdot\mu (d_0^2)$ that
 \beqn 
      \int c u^2 d_0^2 \mathrm{d}\mu   \leqslant  \frac{1}{2} \left(\mathcal{E}[u] + b \mu (u^2)\right), \label{eqcd2u2}
   \eeqn
and then
  \beq
     \mu(u\cdot \mathrm{H}u)
     &=& \mathcal{E}[u] + b \mu (u^2) - c \mu\left(u^2d_0^2\right)\\
     &\geqslant& \mathcal{E}[u] + b \mu (u^2) - c \left( 8C\mathcal{E}[u] + 2 \mu (u^2) \mu (d_0^2) \right)
     \ \geqslant\ \frac{1}{2} \left(\mathcal{E}[u] + b \mu (u^2) \right).
  \eeq
Combining these estimates with the H\"{o}lder inequality yields for any $u\in \mathbb{D}(\mathcal{E})$ and $v\in \mathbb{D}(\mathrm{L})$
   \beq
       (u, \mathrm{H} v) & = & \int \nabla u \cdot \nabla v -cd_0^2uv + buv \;\mathrm{d}\mu \\
       & \leqslant & \left(\mathcal{E}[u]\right)^{\frac{1}{2}} \left(\mathcal{E}[v]\right)^{\frac{1}{2}}  
                                + c \left(\mu(d_0^2u^2)\right)^{\frac12}\left(\mu(d_0^2v^2)\right)^{\frac12} + b \left(\mu(u^2)\right)^{\frac12}\left(\mu(v^2)\right)^{\frac12} \\
       & \leqslant &  \frac52 \left(\mathcal{E}[u] + b \mu (u^2)\right)^{\frac12} \left(\mathcal{E}[v] + b \mu (v^2)\right)^{\frac12}.
   \eeq

Hence, $(u, \mathrm{H}v)$ determines a coercive Dirichlet form, and $\mathrm{H}$ is a positive definite self-adjoint Schr\"{o}dinger operator with its spectrum contained in $(0,\infty)$. It means $\mathrm{H}^{-1}$ exists on $L^2(\mu)$ according to the Lax-Milgram Theorem, i.e. $u=\mathrm{H}^{-1}w\in H^1(\mu)$ (the $L^2$-integrable Sobolev space of weak derivatives of first order) is a weak solution of Equation (\ref{eqSchr}), see Evans \cite{Evans} or Gilbarg-Trudinger \cite{GiTr}.

Whenever $w\geqslant 0$, the weak maximum principle yields $u=\mathrm{H}^{-1}w \geqslant 0$ $\mu$-a.e. too. As a routine, we set $u_{-} = -\min\{u, 0\}$, which has weak derivatives and satisfies
   \beq
     0 \leqslant \mu(u_{-}w) &=& \mu(u_{-}\cdot \mathrm{H}u) \ =\  \int u_{-}\cdot \mathrm{L}u \mathrm{d}\mu - \int \phi u_{-}u  \mathrm{d}\mu\\
      &=&   - \int \nabla u_{-}\cdot \nabla u \mathrm{d}\mu - \mu(\phi u_{-}^2) \\
      &=& - \mathcal{E}[u_{-}] - b\mu(u_{-}^2) + c\mu(d_0^2 u_{-}^2) \ \leqslant\  -\frac12 \left( \mathcal{E}[u_{-}] + b\mu(u_{-}^2) \right) \leqslant 0,
   \eeq
where the middle inequality is derived from (\ref{eqcd2u2}). It follows $\mu(u_{-}^2) =0$, which implies $u_{-} = 0$ $\mu$-a.e. and then $u\geqslant 0$ $\mu$-a.e.

Now fix $w\equiv 1$ and $u = \mathrm{H}^{-1}1$. For any ball $B\subset E$, $u\mathbf{1}_B$ gives a weak solution to Equation (\ref{eqSchr}) restricted in $B$, since there holds for any $h\in C_{\mathrm{c}}^1(B)$ (i.e. the set of first-order derivative functions with compact support in $B$)
    \[ \int_B \mathrm{H}u \cdot h \mathrm{d}x = \int \mathrm{H}u \cdot (he^V) \mathrm{d}\mu = \int w \cdot (he^V) \mathrm{d}\mu = \int_B w \cdot h \mathrm{d}x. \]
When $E$ is the Euclidean space, according to \cite[Theorem 8.22]{GiTr} and the notation therein, $u$ is locally H\"{o}lder continuous in $B$ if we set $f^i=0$, $g=-w$ and $L = \mathrm{L} - \phi$ such that $Lu = g$ as \cite{GiTr} did. Note that the continuity is a local property. Since any local region in Riemannian manifold is (smoothly) diffeomorphic to a region in $\mathbb{R}^n$, which preserves the uniform ellipticity for the (weighted) Laplacian, it follows that $u$ is continuous on $E$ in the framework of Riemannian manifolds.

Moreover, we prove that $u>0$ everywhere. By contradiction, assume $u(y)=0$ at some $y$. Choose $r>0$ and $v\in C^2(E)$ to satisfy
  \[ v\geqslant0, \ v(y)>0, \ v_{|B_r(y)^c} = 0, \ \mathrm{H}v \leqslant \frac12. \]
It follows on $E$
  \[ \mathrm{H}(u-v) = w - \mathrm{H}v \geqslant \frac12, \]
which implies $u-v\geqslant0$ $\mu$-a.e. by the weak maximum principle. Then $u(y)\geqslant v(y)$ by the continuity, which is absurd.

As consequence, we obtain $u\in H^1(\mu)\cap C(E)$ is strictly positive everywhere, and $u^{-1}$ is locally bounded. So $u$ is a Lyapunov function for $\mathrm{L}u \leqslant \phi u$.
\end{proof}

\bigskip

\section{Stability of $\mathrm{W\hspace*{-0.5mm}_2I}$}
 \label{W2IStable}
 \setcounter{equation}{0}

Now we prove Theorem \ref{thmW2IStable}, with similar arguments for Theorem \ref{thmW2ILya} and a few more efforts.

\begin{proof}
Assume $\mu$ verifies $\mathrm{W\hspace*{-0.5mm}_2I}(C)$. Thanks to the implication from $\mathrm{W\hspace*{-0.5mm}_2I}$ to $\mathrm{W\hspace*{-0.5mm}_2H}$ by \cite{GLWY} and the stability of $\mathrm{W\hspace*{-0.5mm}_2H}$ by \cite{GRS}, $\tilde{\mu}$ also verifies the Bobkov-G\"{o}tze's criterion, namely there exists $\tilde{C_T}>0$ such that for any bounded $h$ with $\tilde{\mu} (h) =0$,
    \[ \int e^{Q_{\tilde{C_T}}h} \mathrm{d}\tilde{\mu} \leqslant 1. \]
Recall the first part of proof for Theorem \ref{thmW2ILya}, we introduce $\mathrm{d}\tilde{\nu} = f^2\mathrm{d}\tilde{\mu}$ with $f> 0$ and $\tilde{A}=\{x: 0< f\leqslant N\cdot \tilde{\mu} (f) \}$, and then deal with $\int f^2Q_{\tilde{C_T}} h\mathrm{d}\tilde{\mu}$ by combining three estimates similar as (\ref{eqDecom1}-\ref{eqDecom3}).

But there is a gap since (\ref{eqDecom3}) relies on the Lyapunov condition, which holds for $\mu$ by Theorem \ref{thmW2ILya}, not proved for $\tilde{\mu}$ yet. Nevertheless, it can be quickly fixed due to $\frac{\mathrm{d}\tilde{\mu}}{\mathrm{d}\mu}$ is two-sided bounded so that
   \beqn
      \int_{\tilde{A}^c} f^2Q_{\tilde{C_T}} h \; \mathrm{d}\tilde{\mu}
      &\leqslant& \frac{N^2}{\tilde{C_T}(N-1)^2} \int_{\tilde{A}^c} \left(f-\tilde{\mu} (f)\right)^2 \left(d_0^2 + \mu (d_0^2) \right) \mathrm{d}\tilde{\mu} \nonumber\\
      &\leqslant& \frac{MN^2}{\tilde{C_T}(N-1)^2} \int_{\tilde{A}^c} \left(f-\tilde{\mu} (f)\right)^2 \left(d_0^2 + \mu (d_0^2) \right) \mathrm{d}\mu \nonumber\\
      &\leqslant& \frac{MN^2}{\tilde{C_T}(N-1)^2} \left(\frac{1}{c} \int |\nabla f|^2\mathrm{d}\mu +  \frac{b+ \mu (d_0^2)}{c} \int \left(f-\tilde{\mu} (f)\right)^2 \mathrm{d}\mu \right) \nonumber\\
      &\leqslant& \frac{M^2N^2}{\tilde{C_T}(N-1)^2} \left(\frac{1}{c} \int |\nabla f|^2\mathrm{d}\tilde{\mu} +  \frac{b+ \mu (d_0^2)}{c} \int \left(f-\tilde{\mu} (f)\right)^2 \mathrm{d}\tilde{\mu} \right). \label{eqDecom3tilde}
   \eeqn
Hence, we can substitute (\ref{eqDecom3tilde}) for (\ref{eqDecom3}) and then complete the proof by using the stability of Poincar\'{e} inequality.
\end{proof}

\bigskip

\section{A characterization of $\mathrm{W\hspace*{-0.5mm}_2H}$ via a restricted $\mathrm{W\hspace*{-0.5mm}_2I}$}
 \label{W2HLya}
 \setcounter{equation}{0}
This section has two parts. First, we give a direct proof $\mathrm{W\hspace*{-0.5mm}_2H}$ under the Lyapunov condition, which was originally proposed by \cite[Theorem 1.2]{CGW}. Next, we introduce a restricted $\mathrm{W\hspace*{-0.5mm}_2I}$ interpolating between $\mathrm{W\hspace*{-0.5mm}_2H}$ and the restricted LSI in the sense of \cite[Definition 1.3]{GRS}. We would like to point out here, the theory of Hamilton-Jacobi equations played a fundamental role in studying quadratic transport inequalities. The general work by Gozlan-Roberto-Samson \cite{GRS-HJ} on metric spaces, which extended some early results by Lott-Villani \cite{LV}, is sufficiently adapted in the framework of Riemannian manifolds. 

\subsection{An alternative proof of $\mathrm{W\hspace*{-0.5mm}_2H}$ under the Lyapunov condition (\ref{eqLya})}

Now we prove Theorem \ref{thmW2H}.
\begin{proof}
Let $\delta>0$ be some parameter. Given any bounded $h$ with $\mu (h) =0$, define for all $x\in E$ and $t>0$
  \[\phi(x,t) = \delta t Q_th(x), \ \ \ \Lambda=\mu (e^{\phi}), \ \ \ \lambda = \mu (e^{\phi/2}).  \]
According to \cite{GRS-HJ},  the Hopf-Lax formula $Q_th$ solves the Hamilton-Jacobi equation
  \[ \left\{\begin{array}{l}\frac{\mathrm{d}}{\mathrm{d}t} u + \frac12|\nabla u|^2=0,\\u_0=h,\end{array}\right.\]
and we have
  \beqn
     \frac{\mathrm{d}\Lambda}{\mathrm{d}t} =  \int e^{\phi} \left(\delta Q_th + \delta t \frac{\mathrm{d}Q_th}{\mathrm{d}t} \right) \mathrm{d}\mu
     = \frac{1}{\delta t} \int \delta e^{\phi}\phi - 2|\nabla e^{\phi/2}|^2 \mathrm{d}\mu. \label{eqLambdaDerivative}
  \eeqn

We need to estimate the following
   \[  \int e^{\phi}\phi  \mathrm{d}\mu = \int e^{\phi} \log \frac{e^\phi}{\lambda^2}  \mathrm{d}\mu + \Lambda\log\lambda^2. \]
Set $A = \{x: e^{\phi/2} \leqslant N\lambda\}$ for any given $N>1$, which implies $\lambda^2<e^{\phi}$ on $A^c$. Using Lemma \ref{lemBarthe} on $A$ yields
   \beqn
      \int_A e^{\phi} \log \frac{e^\phi}{\lambda^2}  \mathrm{d}\mu
      &\leqslant& \int_A e^{\phi} - \lambda^2 + (1+N)^2 (e^{\phi/2} -\lambda)^2 \mathrm{d}\mu \nonumber\\
      &=& \int e^{\phi} - \lambda^2 \mathrm{d}\mu + \int_{A^c} \lambda^2 - e^{\phi} \mathrm{d}\mu + (1+N)^2 \int_A (e^{\phi/2} -\lambda)^2 \mathrm{d}\mu \nonumber\\
      &=& \int (e^{\phi/2} -\lambda)^2  \mathrm{d}\mu + \int_{A^c} \lambda^2 - e^{\phi} \mathrm{d}\mu + (1+N)^2 \int_A (e^{\phi/2} -\lambda)^2 \mathrm{d}\mu \nonumber\\
      &\leqslant& \left[1+(1+N)^2\right] \mu(e^{\phi/2} - \lambda)^2. \label{eqW2HEst1}
   \eeqn
Next, let $d_\mu^2(x) = \mu \left(d^2(x,\cdot)\right)$, we have due to the facts $e^\phi \leqslant \frac{N^2}{(N-1)^2}(e^{\phi/2} - \lambda)^2$ on $A^c$ and $\phi \leqslant \frac{\delta}{2}d_\mu^2$ on $E$
   \beqn
      \int_{A^c} e^{\phi} \log \frac{e^\phi}{\lambda^2}  \mathrm{d}\mu
      &=& \int_{A^c} e^{\phi}\phi \mathrm{d}\mu - \int_{A^c}e^{\phi} \mathrm{d}\mu \log\lambda^2  \nonumber\\
      &\leqslant& \frac{\delta N^2}{2(N-1)^2} \int_{A^c} (e^{\phi/2} - \lambda)^2 d_\mu^2 \mathrm{d}\mu - \int_{A^c}e^{\phi} \mathrm{d}\mu \log\lambda^2. \label{eqW2HEst2}
   \eeqn
It follows from Lemma \ref{lemLya} and $d_\mu^2\leqslant 2d_0^2 + 2\mu (d_0^2)$ that
   \beqn
    \ \ \ \ \int_{A^c} (e^{\phi/2} - \lambda)^2 d_\mu^2 \mathrm{d}\mu \leqslant  \frac{2}{c} \int |\nabla e^{\phi/2}|^2 \mathrm{d}\mu + \frac{2b+2\mu (d_0^2)}{c}\int (e^{\phi/2} - \lambda)^2 \mathrm{d}\mu. \label{eqW2HEst22}
   \eeqn
Combining (\ref{eqW2HEst1}-\ref{eqW2HEst22}) with the Poincar\'{e} inequality $\mathrm{PI}(C_P)$ gives
   \beqn
      \int e^{\phi}\phi  \mathrm{d}\mu &\leqslant& C_0 \int |\nabla e^{\phi/2}|^2 \mathrm{d}\mu
        + \int_Ae^{\phi} \mathrm{d}\mu \log\lambda^2,\label{eqW2HEst3}
   \eeqn
where $C_0 = C_P[1+(1+N)^2] + \frac{\delta N^2 \left[1+\left(b+\mu (d_0^2)\right)C_P\right]}{c(N-1)^2}$.

Combining (\ref{eqLambdaDerivative}) and (\ref{eqW2HEst3}), we take $\delta>0$ with $\delta C_0= 2$ so that
  \[ \frac{\mathrm{d}\Lambda}{\mathrm{d}t}
    \leqslant \frac{1}{t} \int_Ae^{\phi} \mathrm{d}\mu \log\lambda^2 \leqslant \frac{1}{t} (\Lambda \log\Lambda) \vee 0. \]
The rest work is similar to the last step in the proof of \cite[Lemma 3.2]{CGW}. Note that $\Lambda(0) = 1$. If $\Lambda(1) > 1$, let $t_0\in [0, 1)$ be the maximal time such that $\Lambda(t_0) = 1$. Then for all $t\in (t_0, 1)$ holds $\frac{\mathrm{d}\Lambda}{\mathrm{d}t} \leqslant \frac{\Lambda \log\Lambda}{t} $, which means $\frac{\mathrm{d}}{\mathrm{d}t} (\frac{\log\Lambda}{t}) \leqslant 0$ and thus
$\log\Lambda(1) \leqslant \lim_{t\downarrow t_0}\frac{\log\Lambda(t)}{t} = 0$. It contradicts the assumption $\Lambda(1) > 1$. Hence, the Bobkov-G\"{o}tze's criterion is verified and $\mathrm{W\hspace*{-0.5mm}_2H}(\frac{1}{\delta})$ follows.
\end{proof}

\bigskip
\subsection{A restricted $\mathrm{W\hspace*{-0.5mm}_2I}$}
According to \cite{GRS, GRS-HJ}, define the supremum-convolution
  \[ P_t h(x) = \sup\limits_{y\in E} \left\{h(y) - \frac{1}{2t}d^2(x,y)\right\}, \]
which solves the Hamilton-Jacobi equation
  \[ \left\{\begin{array}{l}\frac{\mathrm{d}}{\mathrm{d}t} u - \frac12|\nabla u|^2=0,\\u_0=h,\end{array}\right.\]
When $h$ is Lipschitz continuous, there is a unique weak solution to the equation (see also Evans \cite[Section 3.3, Theorem 7]{Evans}), which means $P_{1-t}f = Q_tP_1f$ for $t\in [0,1]$. Moreover, $P_t h$ is $t^{-1}$-semi-convex by \cite[Lemma 5.3]{GRS}.

Now we prove Theorem \ref{thmRestrictedW2I}.

\begin{proof}
If $\mathrm{W\hspace*{-0.5mm}_2H}(C_T)$ holds, we have the $\mathrm{rLSI}(C_T)$ for all $K$-semi-convex functions with $0\leqslant K < C_T^{-1}$ by \cite[Theorem 1.5]{GRS}. Then for $K$-semi-convex $f$ with $\mu (e^f) =1$
  \[ W_2(e^f\mu, \mu)^2 \leqslant 2C_T \mathrm{Ent}_\mu(e^f) \leqslant \frac{16C_T^2}{(1-KC_T)^2}I(e^f), \]
which gives the $\mathrm{rW\hspace*{-0.5mm}_2I}(2C_T)$.

On the other hand, if $\mathrm{rW\hspace*{-0.5mm}_2I}(C)$ holds, we adjust the proof of Theorem \ref{thmW2H} to get $\mathrm{W\hspace*{-0.5mm}_2H}$. Let $\delta>0$, given any bounded Lipschitz function $h$, we introduce for all $x\in E$ and $t\in [0,\frac12]$
  \[\psi(x,t) = \delta t P_{1-t}h(x) - \delta t \mu (P_1h), \ \ \ \Lambda=\mu (e^{\psi}).  \]
It follows due to $P_{1-t}h = Q_{t}P_1h$
  \beqn
     \frac{\mathrm{d}\Lambda}{\mathrm{d}t} &=&  \int e^{\psi} \left( \delta P_{1-t}h + \delta t \frac{\mathrm{d}P_{1-t}h}{\mathrm{d}t} - \delta \mu (P_1h) \right) \mathrm{d}\mu \nonumber\\
     &=& \frac{1}{\delta t} \int \delta^2tQ_t\left(P_1h-\mu (P_1h)\right) e^{\psi} - 2|\nabla e^{\psi/2}|^2 \mathrm{d}\mu. \label{eqLambdaDerivative2}
  \eeqn
Let $g=P_1h-\mu (P_1h)$, since $\psi$ is $\delta$-semi-convex, we have by using $\mathrm{rW\hspace*{-0.5mm}_2I}(C)$ for any $0< \delta \leqslant \frac{1}{2C}$
  \beqn
   \int \delta^2 t Q_tg e^{\psi}\mathrm{d}\mu &=& \delta^2\int Q_1(tg) e^{\psi}\mathrm{d}\mu \nonumber\\
   &\leqslant& \frac{\delta^2 \mu (e^{\psi})}{2} W_2\left(\frac{e^\psi}{\mu (e^{\psi})}\mu, \mu\right)^2 \nonumber\\
   &\leqslant& \frac{2\delta^2C^2}{(1-\delta C)^2} I(e^\psi) \ \leqslant \ 2I(e^\psi). \label{eqEstRW2I}
  \eeqn
Combining (\ref{eqLambdaDerivative2}-\ref{eqEstRW2I}) gives $\frac{\mathrm{d}\Lambda}{\mathrm{d}t}\leqslant 0$, which implies due to $P_{1-t}h \geqslant h$
  \beq
   1 =\Lambda(0) \geqslant \Lambda\left(\frac12\right) &=& \mu \left(\exp\left\{ \frac{\delta}{2} \left(P_{\frac12}h - \mu (P_1h)\right) \right\} \right) \\ 
      &\geqslant& \mu \left( \exp\left\{ \frac{\delta}{2}h - \mu \left(P_{\frac{2}{\delta}} \left(\frac{\delta}{2}h\right)\right) \right\} \right). 
  \eeq
Hence, we obtain $\mathrm{W\hspace*{-0.5mm}_2H}(\frac{2}{\delta})$ by the Bobkov-G\"{o}tze's supremum-convolution criterion (see \cite[Theorem 3.1]{GRS}). Choosing $\delta = \frac{1}{2C}$ gives $\mathrm{W\hspace*{-0.5mm}_2H}(4C)$.
\end{proof}

\subsection*{Acknowledgements}

{\small It is my great pleasure to thank Prof. Li-Ming Wu for his warm encouragement. And I deeply appreciate the anonymous reviewer for his/her conscientious reading and many suggestions on the first version. This work is supported by NSFC (no. 11201456, no. 1143000182, no. 11371352), AMSS research grant (no. Y129161ZZ1), and Key Laboratory of Random Complex Structures and Data, Academy of Mathematics and Systems Science, Chinese Academy of Sciences (No. 2008DP173182).}




%

\end{document}